\documentclass[12pt]{amsart}
\usepackage{amssymb}
\usepackage{parskip}
\usepackage{tikz}
\usepackage{mathtools}
\usepackage{relsize}
\usetikzlibrary{shapes.geometric}
\usetikzlibrary{decorations.markings}

\newtheorem{proposition}{Proposition}[section]
\newtheorem{theorem}[proposition]{Theorem}
\newtheorem{lemma}[proposition]{Lemma}
\newtheorem{corollary}[proposition]{Corollary}

\theoremstyle{definition}
\newtheorem{definition}[proposition]{Definition}
\newtheorem{example}[proposition]{Example}
 
\theoremstyle{remark}
\newtheorem{remark}[proposition]{Remark}

\usepackage{fourier}
\usepackage[T1]{fontenc}
\title[Homotopy type of the flag complex over a finite vector space]{The homotopy type of skeleta of the flag complex over a finite vector space}
\thanks{The first and third authors were supported by Conacyt scholarships. The second author was supported by Conacyt research grant CB-2013-01-221221.}
\author{Jorge Aguilar-Guzm\'an, Jes\'us Gonz\'alez and Jos\'e Luis Le\'on-Medina}
\usepackage[hidelinks]{hyperref}

\begin{document}
\maketitle

\begin{abstract}
The aim of this paper is to give a (discrete) Morse theoretic proof of the fact that the $k$-th skeleton of the flag complex $\mathcal{F}$, associated to the lattice of subspaces of a finite dimensional vector space, is homotopy equivalent to a wedge of spheres of dimension $\min\{k,\dim(\mathcal{F})\}$. The tight control provided by Morse theoretic methods allows us to give an explicit formula for the number of spheres appearing in each of these wedge summands.
\end{abstract}


\section{Introduction}

One of the most impressive applications of discrete Morse theory is the availability to determine the homotopy type of a simplicial complex by constructing a suitable discrete Morse function. An illustrative example of such a situation is the determination of the homotopy type of the flag complex $F(V)$ associated to the lattice of subspaces of a finite vector space $V$ (see the next section for a review of the explicit definitions). As proved in \cite[Proposition 3.6]{Kratzer-Thevenaz}, $F(V)$ is homotopy equivalent to a wedge of spheres. A Morse theoretic proof of such a fact is indicated in \cite{Zax} where, however, some of the key proof details are not provided. A first aim of this paper is to clarify and formalize some of the ideas in \cite{Zax}, providing complete proof details. Additionally, and also following the indications in \cite{Zax}, we prove in detail the corresponding homotopy equivalence
\begin{equation}\label{decomposition}
F(V)^{(k)}\simeq\bigvee S^{\min\{k,\dim(F(V))\}}.
\end{equation}

The fact that each skeleton of $F(V)$ has the homotopy type of a wedge of spheres is certainly well known, for $F(V)$ is shellable, and skeleta of shellable complexes are shellable again (see~\cite[Theorem 8.2.18]{MR2724673}). A main contribution of this work is to show that, by replacing shellability methods by the fine control coming from discrete Morse theory techniques, it is possible to derive an explicit formula for the number of spheres appearing in~(\ref{decomposition}) ---information that, to the best of our knowledge, was not available previously. This illustrates the general principle noted in~\cite[Remark~12.4]{MR2361455}.

We will follow the standard notation and conventions in Forman's discrete Morse theory, see for example \cite{Forman}. The key result we need is the following:

\begin{theorem}[{\cite[Theorem 2.5]{Forman}}]\label{homotopytype} Let $X$ be a simplicial complex with a discrete Morse function $f$. Then $X$ is homotopy equivalent to a cell complex containing the same number of cells of a given dimension as there are critical simplices of $f$ of that dimension.
\end{theorem}

It is usually more convenient to handle acyclic pairings rather than discrete Morse functions, as the former ones ignore irrelevant quantitative information of the latter ones, preserving only the essential qualitative information regarding critical simplices. Pairings coming from discrete Morse functions are characterized by and can be constructed using the following results.

\begin{theorem}[{\cite[Theorem 3.5]{Forman}}]\label{theorem2} A pairing $P$ on a simplicial complex $X$ corresponds to a discrete Morse function $f$ if and only if every simplex of $X$ appears in at most one pair of $P$ and $P$ is acyclic.
\end{theorem}

\begin{lemma}[{\cite[Lemma~4.2]{jonsson}}]\label{lemmapairing} Let $X$ be a simplicial complex that decomposes as the disjoint union of non-empty collections $X_i$ of simplexes, indexed by the elements $i$ in a partially ordered set $I$. Assume that for each $i\in I$, $\bigcup_{j \leq i} X_j$ is a subcomplex of $X$. For each $i\in I$, let $P_i$ be an acyclic pairing on the simplices of $X_i$. Then $\bigcup_{i\in I} P_i$ is an acyclic pairing on~$X$.
\end{lemma}

Lemma~\ref{lemmapairing} is taken from~\cite[p.~27]{Zax}, where it appears with the hypothesis that $I$ has a unique minimal element. The additional condition is reminiscent from the lexicographic discrete Morse function constructions in \cite[Lemma~4.1]{hersh}, which is the source reference used by Zax. Lemma~\ref{lemmapairing} is easily deduced from \cite[Lemma~4.2]{jonsson}; alternatively, either of the (essentially equivalent) proofs given in~\cite{hersh,jonsson,Zax} works for our purposes.

\section{The Flag Complex}

\begin{definition} Let $V$ be an $n$-dimensional vector space. A $k$-flag $f_k$ in $V$ is a sequence of $k$ subspaces $V_{d_1}, \dots , V_{d_k}$ such that $0 \subsetneq V_{d_1}\subsetneq \cdots \subsetneq V_{d_k}\subsetneq V$. Here the index $d_j$ in $V_{d_j}$ stands for the dimension of $V_{d_j}$, i.e., $d_j=\dim (V_{d_j})$. By abuse of notation, we will write $f_k=V_{d_1}\subsetneq\cdots\subsetneq V_{d_k}$, while the notation $W\in f_k$ will mean that $W=V_{d_j}$ for some $j=1,\ldots,k$. Further, the ($k-1$)-flag obtained from $f_k=V_{d_1}\subsetneq\cdots\subsetneq V_{d_k}$ by removing the subspace $V_{d_\ell}$ will be denoted by $f_k\setminus V_{d_\ell}$, while the inverse operation (inserting a subspace in a chain of nested subspaces) will be indicated by a plus sign. Thus, in the situation above, $f_k=(f_k\setminus V_{d_\ell})+V_{d_\ell}$.

Let $\mathbb{F}_q$ be a finite field, the flag complex $F(\mathbb{F}_q^n)$  is the (abstract) simplicial complex whose vertices are the proper subspaces of $\mathbb{F}_q^n$ and whose ($k-1$)-dimensional simplices are the $k$-flags of $\mathbb{F}_q^n$. Note that face relation is given by taking subsequences.
\end{definition}

\begin{example}\label{exampleF23} \emph{The complex $F(\mathbb{F}_2^3)$.}
Denote the three standard basis vectors of $\mathbb{F}_2^3$ as $e_1=(1{,}\ 0{,} \ 0)$, $e_2=(0{,}\ 1{,}\ 0)$ and $e_3=(0{,}\ 0{,}\ 1)$. The vector space spanned by vectors $w_1, \dots, w_n$ is denoted by $\langle w_1, \dots, w_n \rangle$. Note that $\mathbb{F}_2^3$ has seven lines $$\mbox{$\langle e_1 \rangle$, $\langle
e_2\rangle$, $\langle e_3 \rangle$, $\langle e_1+e_2\rangle$, $\langle e_1+e_3\rangle$, $\langle e_2+e_3\rangle$, $\langle e_1+e_2+e_3\rangle$}$$ and seven planes
\begin{align*}
\langle e_1, e_2\rangle, & \langle e_1, e_3 \rangle, \langle e_1, e_2+e_3\rangle, \langle e_2, e_3 \rangle, \\ & \langle e_2, e_1+e_3\rangle, \langle e_3, e_1+e_2\rangle, \langle e_1+e_2, e_1+e_3\rangle.
\end{align*}
Altogether, there are 14 vertices in $F(\mathbb{F}_2^3)$. The 1-simplices of $F(\mathbb{F}_2^3)$ are those sequences $L\subsetneq P$ consisting of a line $L$ contained in a plane $P$. Although a plane is determined by two lines, it contains a total of three lines. So, in total, there are 21 simplices of dimension 1. An alternative way of counting 1-simplexes is by observing that each line is contained in three different planes. The complete simplicial structure of $F(\mathbb{F}^3_2)$ can be represented by the Heawood graph shown in Figure~\ref{heawood}.
\begin{figure}[htb]
\begin{center}\footnotesize
\begin{tikzpicture}
\node[draw=black,ultra thick, minimum size=6cm,regular polygon,regular polygon sides=14] (a) {};
\foreach \x in {1,2,...,14}
  \fill (a.corner \x) circle[radius=4pt];
\draw[ultra thick] (a.corner 12)node[right, xshift=0.1cm]{$\langle e_1+e_2+e_3\rangle$} -- (a.corner 3)node[left, yshift=0.3cm]{$\langle e_1, e_2+e_3\rangle$};
\draw[ultra thick] (a.corner 13)node[right, yshift=0.2cm] {$\langle e_2, e_1+e_3\rangle$}--(a.corner 8)node[below,xshift=-0.2cm, yshift=-0.08cm]{$\langle e_1+e_3\rangle$};
\draw[ultra thick] (a.corner 14)node[above, xshift=0.4cm]{$\langle e_2 \rangle$}--(a.corner 5)node[left, xshift=-0.1cm]{$\langle
 e_2, e_3 \rangle$};
\draw[ultra thick] (a.corner 1)node[above, yshift=0.1cm]{$\langle e_1,e_2 \rangle$}--(a.corner 10)node[below,xshift=0.4cm, yshift=-0.1cm]{$\langle
 e_1+e_2\rangle$};
\draw[ultra thick] (a.corner 2)node[above, yshift=0.12cm]{$\langle e_1 \rangle$}--(a.corner 7)node[below, xshift=-0.4cm, yshift=-0.1cm]{$\langle
 e_1,e_3 \rangle$};
\draw[ultra thick] (a.corner 4)node[left, xshift=-0.1cm, yshift=0.1cm]{$\langle e_2+e_3\rangle$}--(a.corner 9)node[below, xshift=0.4cm, yshift=-0.08cm]{$\langle e_1+e_2, e_1+e_3\rangle$};
\draw[ultra thick] (a.corner 6)node[left, xshift=-0.08cm, yshift=-0.05cm]{$\langle e_3 \rangle$}--(a.corner 11)node[right, xshift=0.08cm]{$\langle e_3, e_1+e_2 \rangle$};
\end{tikzpicture}
\end{center}
\caption{Simplicial structure of $F(\mathbb{F}_2^3)$}\label{heawood}
\end{figure}
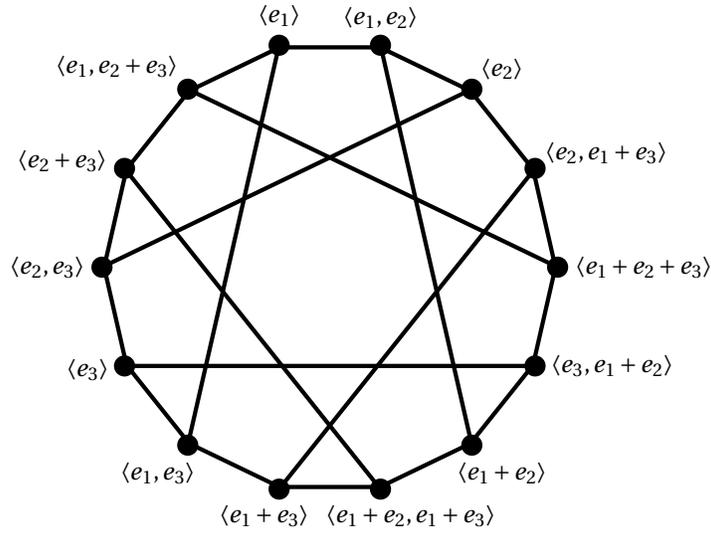
\end{example}

We analyze the homotopy type of the flag complex $F(\mathbb{F}_q^n)$ and of each of its skeleta $F(\mathbb{F}_q^n)^{(k)}$ using discrete Morse theory. In summary, we will proceed as follows:
\begin{itemize}
\item[(A)] First we label each maximal flag in order to induce a partition of the simplices of the complete flag complex.
\item[(B)] Next we give an acyclic pairing for each class in the above partition, and use Lemma~\ref{lemmapairing} to get a corresponding acyclic pairing for $F(\mathbb{F}^n_q)$.
\item[(C)] Then we use Theorem~\ref{theorem2} to prove that the homotopy type of the complex $F(\mathbb{F}^n_q)$ is a wedge of $(n-2)$-spheres. Counting the number of critical simplices in the pairing gives us the number of spheres in the wedge sum.
\item[(D)] Lastly, we delete all those simplices of dimension greater than $k$ and count the number of the resulting critical simplices. This gives us the number of spheres in the homotopy description of the skeleton $F(\mathbb{F}_q^n)^{(k)}$ as a wedge of $k$-spheres.
\end{itemize}

\section{Partitioning the Flag Complex}\label{secpar}

Let $\Sigma_n$ denote the set of permutations of the first $n$ natural numbers. We spell out a permutation $\sigma\in\Sigma_n$ by yuxtaposition of its values: $\sigma=\sigma(1)\sigma(2)\cdots\sigma(n)$.

In this section we assign, to each $k$-flag $f_k$, both an $n\times n$ matrix with entries in $\mathbb{F}_q$, and a label in $\Sigma_n$. Details are given below, first when $f_k$ is maximal (i.e.~$k=n-1$):

We use elementary column operations which involve either multiplying a column by a non-zero factor, or adding a multiple of some column $i$ to some other column $j$ with $i<j$, in order to associate, to each maximal flag $f_{n-1}=V_1 \subsetneq V_2 \subsetneq \cdots \subsetneq V_{n-1}$ in $F(\mathbb{F}_q^n)$, an $n\times n$ matrix $[a_{ij}]$ satisfying:
\begin{itemize}
\item[\it (i)] The first $k$ vector columns $w_1, \dots, w_k$ of $[a_{ij}]$ span the $k$-dim\-en\-sional vector space $V_k$ for each $k=1,\ldots, n$. (Here and in what follows we agree to set $V_{n}=\mathbb{F}_q^n$.)
\item[\it (ii)] For each column $j$, the highest row value $i$ with $a_{ij}\neq 0$ has in fact $a_{ij}=1$. Under these conditions, the element $a_{ij}$ is called the pivot of the $j$-th column.
\item[\it (iii)] The matrix $[a_{ij}]$ has zero entries  to the right of each pivot.
\end{itemize}

Note that \emph{(i)} and \emph{(ii)} ensure the uniqueness of $w_1$. Having fixed $w_{1}$, \emph{(i)--(iii)} then imply the uniqueness of $w_2$, and so on. Therefore, the requirements \emph{(i)--(iii)} \rule{.1mm}{0mm}allow us to assign to $f_{n-1}$ a well-defined matrix $M(f_{n-1})$, which will be called the \emph{minimal matrix representation} (or \emph{matrix representation}, for short) of $f_{n-1}$. The label assigned to $f_{n-1}$ (and to $M(f_{n-1})$) is the permutation $i=i_1 i_2 \dots i_n \in \Sigma_n$, where $i_j$ is the row index of the pivot of the $j$-th column in the minimal matrix representation of $f_{n-1}$.
  
Note that maximal flags are recovered from their matrix representations, however different maximal flags can have the same label.

\begin{example}\label{ejemplominimalrepresentation} Consider the maximal flag $f_{2}=\langle e_1+e_3 \rangle \subsetneq \langle e_1, e_3 \rangle$ in the vector space $F(\mathbb{F}_2^3)$. Evidently, the first column vector of the matrix representation of $f_2$ must be $w_{1}=(1,\ 0,\ 1)$. We have two choices for the second vector: either $e_1$ or $e_3$. But \emph{(iii)} above rules out $e_3$, so the second column of the matrix representation must be $w_{2}=(1,\ 0,\ 0)$. Likewise, \emph{(iii)} then forces $w_{3}=(0,\ 1,\ 0)$. In this case, the pivots are at positions $(3,1)$, $(1,2)$ and $(2,3)$, and the matrix representation of $f_{2}$ is
\begin{equation}\label{ejemrepre}
\begin{pmatrix}
1 & {\bf 1} & 0 \\
0 & 0 & {\bf 1} \\
{\bf 1} & 0 & 0 
\end{pmatrix},
\end{equation}
with associated label $312$ coming from the boldface ones at the pivot positions.
\end{example}

As indicated above, the construction of the minimal matrix representation of a maximal flag can be done algorithmically via column operations on matrices. For instance, in Example~\ref{ejemplominimalrepresentation} we could start, say, with the basis $w'_1=e_1+e_3$, $w'_2=e_3$, $w'_3=e_1+e_2$ which satisfies \emph{(i)} above. Then~(\ref{ejemrepre}) is obtained from the matrix with columns $w'_1,w'_2,w'_3$ by the following pair of column operations ---each coming from the corresponding observation in Example~\ref{ejemplominimalrepresentation} about the uniqueness of the basis elements $w_2$ and $w_3$:
$$
\begin{pmatrix} 1 & 0 & 1 \\ 0 & 0 & 1 \\ 1\begin{picture}(0,0)\put(-2.5,-4){\qbezier(0,0)(7,-5)(14,0)} \put(13,-2){\vector(1,1){0}}\end{picture} & 1 & 0 \end{pmatrix}\leadsto
\begin{pmatrix} 1 & 1\begin{picture}(0,0)(1,-15)\put(-2.5,-4){\qbezier(0,0)(7,5)(14,0)} \put(13,-6){\vector(1,-1){0}}\end{picture} & 1 \\ 0 & 0 & 1 \\ 1 & 0 & 0 \end{pmatrix}\leadsto
\begin{pmatrix} 1 & 1 & 0 \\ 0 & 0 & 1 \\ 1 & 0 & 0 \end{pmatrix}
$$

As every simplex in the flag complex $F(\mathbb{F}_q^n)$ is a face of some maximal flag, we can partition the set of simplices in the flag complex by considering the first time a simplex appears as a face of a maximal flag, according to the lexicographic order in $\Sigma_n$. Explicitly:

For each label $i$, let $f_{n-1}^i$ denote any maximal flag that has label $i$. We define the set
\begin{multline*}
X_i = \left\lbrace f_k \in F(\mathbb{F}_q^n) \,\big\lvert\,  
 \text{ there is some } f_{n-1}^i \text{ with } f_k\subseteq f_{n-1}^i,\right.\\\left. \text{ but } f_k \nsubseteq f_{n-1}^j \text{ for any }f_{n-1}^j \text{ with } j<i
\right\rbrace.
\end{multline*}
We will prove that, for a flag $f_k$, the label $i$ with $f_k\in X_i$ can be explicitly described by the procedure below. Indeed, we will see in fact that the procedure actually describes the minimal matrix representation of the first maximal flag containing $f_{k}$ as a face.

Let $f_k=V_{d_1} \subsetneq V_{d_2} \subsetneq \cdots \subsetneq V_{d_k}$ where $\dim V_{d_j} = d_j$, then{:}
\begin{itemize}
\item[\it (a)] Select a set of $n$ linearly independent vectors $v_1, \dots, v_n$ such that $v_1, \dots v_{d_j}$ span the vector space $V_{d_j}$ for $j=1,\cdots, k+1$. (Here we set $d_{k+1}=n$, so that $V_{d_{k+1}}= \mathbb{F}_q^n$).  
\item[\it (b)] Let $w_j$ ($1\leq j\leq n$) be the $j$th column of the matrix representation of the maximal flag\rule{.2mm}{0mm} $g_{n-1}=W_1\subsetneq\cdots\subsetneq W_{n-1}$, where $W_k=\langle v_1,\ldots, v_k\rangle$, for $k=1,\ldots,n-1$. Divide the columns of $M(g_{n-1})$ into $k+1$ blocks so that columns in the first $j$ blocks span $V_{d_j}$ for each $j=1$, $\dots$, $k+1$:
\[
\quad\Big[ 
w_{1} \dots w_{d_1} {\big\lvert} w_{d_1+1} \dots w_{d_2} {\big\lvert} \quad\cdots\quad {\big\lvert} w_{d_{k-1}+1} \dots w_{d_k} {\big\lvert} w_{d_k+1} \dots w_n \Big]
\]
\item[\it (c)] By applying elementary column operations within blocks (which does not change the spanned vector spaces $V_{d_j}$), we can go further and produce zeros on the entries to the left, and within the same block, of each pivot. Finally, we reorder the columns within each block so that pivots appear from top to bottom, i.e.~so that, for a pair of consecutive columns in a common block, the row index of the pivot for the column on the right is larger than the row index of the pivot for the column on the left.
\end{itemize}
We now prove that the resulting matrix does not depend on the vectors $v_1,\ldots, v_n$ chosen at step \emph{(a)}.

\begin{proposition}\label{uniquenessminimal}
Let $u_1, \ldots, u_n$ be another set of vectors satisfying (a) and let $v'_1, \ldots, v'_n$ and $u'_1, \ldots, u'_n$ be the column vectors of the respective matrices given by the above procedure. Then $v'_{j}= u'_{j}$ for each $j=1, \ldots, n$.
\end{proposition}
\begin{proof} 
Assume inductively that $v'_k=u'_k$ for $k=1,\ldots,j-1$ (the induction starts with $j=1$ whose hypothesis is vacuously true). Say $v'_{j}$ and $u'_{j}$ lie in the $\ell$-th block of their corresponding matrices, i.e., $v'_{j}, u'_{j}\in V_{d_{\ell}}$.
Note that $V_{d_{\ell}}=\langle u'_1, \dots, u'_{d_{\ell-1}}, u'_{d_{\ell-1}+1}, \ldots, u'_j, \ldots, u'_{d_\ell} \rangle$, so there exists coefficients $\alpha_1, \ldots, \alpha_{d_{\ell}} \in~~ \mathbb{F}_{q}$ such that 
\begin{eqnarray*}
\nonumber
v'_{j}&=&\alpha_{1}u'_{1} + \cdots + \alpha_{d_{\ell-1}}u'_{d_{\ell-1}} + \cdots + \alpha_{j-1}u'_{j-1}+\alpha_{j}u'_{j} + \cdots +\alpha_{d_\ell}u'_{d_{\ell}}\\
&=&\alpha_{1}v'_{1} + \cdots + \alpha_{d_{\ell-1}}v'_{d_{\ell-1}} + \cdots + \alpha_{j-1}v'_{j-1}+ \alpha_{j}u'_{j} + \cdots +\alpha_{d_\ell}u'_{d_{\ell}}. 
\end{eqnarray*}
Since there are zeros to the right of the pivot of each of the vector columns $u'_1=v'_1, \ldots, u'_{j-1}=v'_{j-1},$ we recursively get $\alpha_1=\cdots =\alpha_{j-1}=0$. The resulting simplified equality \begin{equation}\label{simpliequali} v'_{j}= \alpha_{j}u'_{j} + \cdots +\alpha_{d_\ell}u'_{d_{\ell}},\end{equation} and the fact that pivots within blocks have been ordered from top to bottom then imply that the row index of the pivot of $v'_{j}$ cannot be smaller than the row index of the pivot of $u'_{j}$. But the roles of the $u'_k$'s and the $v'_k$'s can be interchanged in the argument, so that in fact the row indexes of the pivots of $u'_j$ and $v'_j$ agree. In these conditions,~(\ref{simpliequali}) and the indicated ordering of pivots further yield $\alpha_{j+1}=\cdots =\alpha_{d_{\ell}}=0$. In summary, $v'_{j}=\alpha_{j}u'_{j}$. 
Finally, $\alpha_{j}=1$ so that $v'_{j}= u'_{j}$, because pivots have been normalized to have value $1$.
\end{proof}

The matrix $M(f_{k})$ obtained by the procedure in steps \emph{(a)--(c)} is thus well-defined, and will be called the \emph{minimal matrix representation} of $f_k$. As in the case of a maximal flag, the label produced by the pivot positions in the resulting matrix yields the label we associate to $f_k$ and to $M(f_{k})$.

If the initial flag $f_k$ were maximal, then step \emph{(c)} above would be vacuous and the process in \emph{(a)--(b)} would reduce to the process in \emph{(i)--(iii)}. In other words, the process in \emph{(a)--(c)} generalizes the process in \emph{(i)--(iii)}. This justifies the fact that we have used the same name for the matrices (and labels) given by both processes. 

\begin{remark}\label{minmatrep} Let $f_{k}$ and $f_{n-1}$ be flags (the latter one being maximal) with $f_{k}$ a face of $f_{n-1}$, and let $g_{n-1}$ be the maximal flag determined by $M(f_k)$, i.e., the maximal flag whose $j$-th vertex is the vector space spanned by the first $j$ vector columns of $M(f_{k})$. It is obvious that $f_k$ is also a face of $g_{n-1}$. A main goal (Theorem~\ref{minimalidad} below) of this section is to prove that the label associated to $g_{n-1}$ is no larger (and even strictly smaller) than the label associated to $f_{n-1}$ (as long as $f_{n-1}\neq g_{n-1}$).
\end{remark}

\begin{example}
For the $2$-flag $f_2=\langle e_1+e_3\rangle \subsetneq \langle e_2+e_3, e_1+e_2, e_2 \rangle$ in $F(\mathbb{F}_2^4)$, the process in \emph{(a)--(c)\rule{.5mm}{0mm}} can start with $w_1=e_1+e_3$, $w_2=e_2+e_3$, $w_3=e_2$ and $w_4=e_4$, so that the matrix coming from the step \emph{(b)\rule{.5mm}{0mm}} is 
\[\left(\begin{array}{c|cc|c}
1 & 1 & \bf{1} & 0 \\
0 & \bf{1} & 0 & 0 \\
\bf{1} & 0 & 0 & 0 \\
0 & 0 & 0 & \bf{1}
\end{array}\right),
\]
For step \emph{(c)} we only need to work on the second block. First we produce a zero entry to the left of the pivot with position $(1,3)$ by adding the third column to the second one to get
\[\left(\begin{array}{c|cc|c}
1 & 0 & \bf{1} & 0 \\
0 & \bf{1} & 0 & 0 \\
\bf{1} & 0 & 0 & 0 \\
0 & 0 & 0 & \bf{1}
\end{array}\right).
\]
Second, we interchange the columns of the second block to get the matrix associated to $f_{2}$:
\[\left(\begin{array}{c|cc|c}
1 & 1 & 0 & 0 \\
0 & 0 & 1 & 0 \\
1 & 0 & 0 & 0 \\
0 & 0 & 0 & 1
\end{array}\right).
\]
The corresponding associated label is 3124. Likewise, the reader can easily check that the matrix associated to the $2$-flag $\varphi_2=\langle e_2, e_3\rangle \subsetneq \langle e_3, e_2, e_1+e_2 \rangle$ in $F(\mathbb{F}_2^4)$ is
\[
\left(\begin{array}{cc|c|c}
0 & 0 & 1 & 0 \\
1 & 0 & 0 & 0 \\
0 & 1 & 0 & 0 \\
0 & 0 & 0 & 1
\end{array}\right),
\]
with label 2314. We next explain in general terms that in fact $f_2\in X_{3124}$ and $\varphi_{2}\in X_{2314}$.
\end{example}

\begin{remark}\label{chica} Let $f_{k}$ and $f_{n-1}$ be as in Remark~\ref{minmatrep}. Start the process for constructing $M(f_{k})$ directly at step \emph{(c)} with the matrix representation $M(f_{n-1})$ of $f_{n-1}$. The operations needed in the first half of step \emph{(c)} do not change the distribution of pivots within blocks of $M(f_{n-1})$, while the operations needed in the second half of step \emph{(c)} can only decrease the labeling of each block. Therefore the label associated to $M(f_k)$ is less than or equal to the label associated to $M(f_{n-1})$. In other words, the label associated to $f_{k}$ is less than or equal to the label associated to $f_{n-1}$.
\end{remark}

\begin{remark}\label{basicamentelaunicidad} In Remark~\ref{chica}, if pivots within blocks of $M(f_{n-1})$ were already ordered from top to bottom, then no additional operations would be needed in order to attain the goal of step \emph{(c)}, so that $M(f_k)=M(f_{n-1})$, and consequently $g_{n-1}=f_{n-1}$, in the notation of Remark~\ref{minmatrep}.
\end{remark}

\begin{theorem}\label{minimalidad}
Let $f_k$ and $g_{n-1}$ be as in Remark~\ref{minmatrep}. Then the label $i\in \Sigma_{n}$ associated to $g_{n-1}$ (and $f_k$) satisfies $f_{k}\in X_{i}$. Indeed, not only is $f_{k}$ a face of $g_{n-1}$, but the label associated of $g_{n-1}$ is strictly smaller than the label associated of any other maximal flag having $f_{k}$ as a face.
\end{theorem}
\begin{proof} 
As noticed in Remarks~\ref{minmatrep} and~\ref{chica}, $f_k$ is a face of $g_{n-1}$, and the label $i\in\Sigma_n$ associated to these two flags is less than or equal to the label $j\in\Sigma_n$ associated of any other maximal flag $f_{n-1}$ having $f_{k}$ as a face. The next result shows that, in fact, $i<j$ whenever $f_{n-1}\neq g_{n-1}$.   
\end{proof}

\begin{proposition}\label{commonface}
Two different maximal flags with the same label $i\in\Sigma_n$ do not have a common face which lies in $X_i$.
\end{proposition}
\begin{proof}
Suppose that $f_{n-1}$ and $g_{n-1}$ are maximal flags sharing label $i\in\Sigma_n$ as well as a face $f_k$ which lies in $X_i$. Then pivots within blocks of both $M(f_{n-1})$ and $M(g_{n-1})$ (in the division in $k+1$ blocks as in step \emph{(b)} of the construction of $M(f_k)$) are forced to be ordered from top to bottom (otherwise $f_k$ would lie in a $X_j$ with $j<i$). As observed in Remark~\ref{basicamentelaunicidad}, this means $M(f_{n-1})=M(f_k)=M(g_{n-1})$, so $f_{n-1}=g_{n-1}$.
\end{proof}

We close this section by capturing the extent to which an $X_i$ fails to be a subcomplex of $F(\mathbb{F}^n_q)$. The resulting characterization plays a key role in the identification of the main properties of the gradient vector field described in the next section.

\begin{proposition}\label{removing} Let $f_{n-1}=V_1\subsetneq V_2 \subsetneq \cdots \subsetneq V_{n-1}$ be a maximal flag with label $i=i_1 i_2 \cdots i_n\in\Sigma_n$. Then the face of $f_{n-1}$ obtained by deleting any set of vertices $V_{\ell_1},\ldots,V_{\ell_m}$ from $f_{n-1}$ lies in $X_i$ if and only if
\begin{equation}\label{condimenor}
i_{\ell_j}<i_{\ell_j+1}
\end{equation}
{for each $j=1,...,m$}. 

\begin{remark}
We are using the term ``face'' for any \emph{non-empty} subset of a simplex. Thus, even if $i$ is taken in
Proposition~\ref{removing} as the minimal label $12\cdots n$ (in which case~(\ref{condimenor}) always holds), it is implicitly assumed that not all vertices of $f_{n-1}$ are to be removed. Actually, the special case of the label $12\cdots n$ plays a subtle role in identifying critical cells of the gradient field that will be constructed in the next section.
\end{remark}
\begin{proof}[Proof of Proposition~\ref{removing}]
For $j\in\{1,\ldots,n\}$ let $w_j$ denote the $j$th column vector of $M(f_{n-1})$. We prove the sufficiency of~(\ref{condimenor}) by induction over the number of vertices removed. At the start of the induction, where we remove only one vertex, say $V_{\ell_1}$, the minimal matrix representation for the resulting face is obtained from 
\[
\Big[ 
w_{1}\big\lvert \cdots \big\lvert w_{\ell_1-1} \big\lvert w_{\ell_1} w_{\ell_1+1} \big\lvert \cdots \big\lvert w_n
\Big],
\]
the matrix representation of $f_{n-1}$, after partitioning it into the $n-1$ indicated blocks. Since $i_{\ell_1}<i_{\ell_1+1}$, the only block of two vectors has its pivots already ordered from top to bottom, so the conclusion follows directly from Remark~\ref{basicamentelaunicidad}. The inductive step is completely similar. Suppose we have a set of $m+1$ vertices $V_{\ell_1},\ldots,V_{\ell_{m+1}}$ to be removed that satisfy~(\ref{condimenor}). After removing the last $m$ vertices, we end up with a matrix of the form
\[
\Big[ 
w_{1}\big\lvert \cdots \big\lvert w_{\ell_1-1} \big\lvert w_{\ell_1} \big\lvert B_1 \big\lvert \cdots \big\lvert B_d
\Big]
\]
where, by induction, pivots within blocks $B_j$ (some of which could consist of only one vector) are ordered from top to bottom. Removing the vertex $V_{\ell_1}$ merges the 1-column block $w_{\ell_1}$ with the block $B_1$:
\begin{equation}\label{readoff}
\Big[ 
w_{1}\big\lvert \cdots \big\lvert w_{\ell_1-1} \big\lvert w_{\ell_1} B_1 \big\lvert \cdots \big\lvert B_d\Big].
\end{equation}
But pivots in $B_1$ are ordered from top to bottom, so that the assumption $i_{\ell_1}<i_{\ell_1+1}$ implies that the corresponding condition on pivots also holds for the new block $w_{\ell_1} B_1$. Therefore the conclusion follows again from Remark~\ref{basicamentelaunicidad}.

The reciprocal follows from Theorem~\ref{minimalidad}: any face of $f_{n-1}$ in $X_i$ has the same minimal matrix representation as that of $f_{n-1}$. 
\end{proof}
\end{proposition}

\begin{corollary}\label{corollary1} If $i$ is the maximal label, then $X_i$ is the set of all those maximal flags with label $i$.
\begin{proof}
Condition~(\ref{condimenor}) never holds for the maximal label $n\, (n-1)\, \cdots \,2\, 1$. So no proper face of a maximal flag with label $i$ can lie in $X_i$, in view of Proposition~\ref{removing}.
\end{proof}
\end{corollary}

\section{The Pairing}

We start by noticing that the hypothesis in Lemma~\ref{lemmapairing} holds for the partition $\{X_i\}_{i\in \Sigma_n}$. We then construct a suitable acyclic pairing for each $X_i$.
The following result is an immediate consequence of Theorem~\ref{minimalidad} (and, in fact, of the definition of the collections $X_i$).
\begin{lemma}\label{subcx}
For each $i\in \Sigma_n$, $\bigcup_{j \leq i} X_j$ is a subcomplex of $X$.
\end{lemma}

For a label $i=i_{1}i_{2}\cdots i_{n}\in \Sigma_{n}$ different from the maximal label $n(n-1)\cdots 1$ let $j_{i}$ stand for the smallest integer $t\in \{1,2,\ldots n-1\}$ such that $i_{t}<i_{t+1}$. In addition, for a maximal flag $f_{n-1}$ with label $i$, let $V_{f_{n-1}}$ be the $j_{i}$-th vertex of $f_{n-1}$, i.e.~the vector space spanned by the first $j_{i}$ columns in the matrix representation of $f_{n-1}$.

\begin{proposition}\label{vertex}
For each label $i\in \Sigma_{n}$ which is neither minimal nor maximal, and for each maximal flag $f_{n-1}$ with label $i$, the 1-flag $V_{f_{n-1}}$ lies in $X_{k}$, for a label $k\in \Sigma_{n}$ with $k<i$.
\end{proposition}
\begin{proof}
Let $k\in \Sigma_{n}$ be the label associated to the minimal matrix representation of the 1-flag $V_{f_{n-1}}$. The latter matrix has the form
\[
\Big[ 
B_1 \big\lvert B_2
\Big].
\]
We have $V_{f_{n-1}}\in X_{k}$ in view of Theorem~\ref{minimalidad}, and we need to check that $k<i $. Put $i=i_{1}i_2\cdots i_n$ and $k=k_{1}k_2\cdots k_n$. Since the columns of $B_{1}$ are a basis of $V_{f_{n-1}}$, and since the pivots in each block $B_i$ are ordered from top to bottom, while the pivots in the first $j_i$ columns of the matrix representation of  $f_{n-1}$ are ordered from bottom to top, we actually have $i_{1}=k_{j_i}>k_{j_{i}-1}>\cdots >k_1$. Thus $i_1>k_1$ provided $j_{i}>1$,
in which case $i>k$. We can therefore assume $j_{i}=1$, so that $i_{1}=k_{1}$. Assume further, for a contradiction, that 
\begin{equation}\label{etiqueta}
i\leq k.
\end{equation}
In particular $i_2\leq k_2$. Since $k_2<k_3<\cdots<k_n$ and $i_1<i_2$ (by definition of $j_i$), we have in fact that $k_1=i_1<i_2\leq k_2<k_3<\cdots<k_n$, forcing $k$ to be the minimal label $12\cdots n$. This contradicts~(\ref{etiqueta}), since $i$ is not the minimal label.
\end{proof}

Let $L_i$ denote the set of maximal flags having label $i$. Proposition~\ref{commonface} implies
$$
X_i=\coprod_{f_{n-1}\in L_i} X_{i,f_{n-1}},
$$
where $X_{i,f_{n-1}}$ consists of the faces of $f_{n-1}$ lying on $X_{i}$. Therefore, in order to construct an acyclic pairing $P_{i}$ on  $X_{i}$, it suffices to construct corresponding acyclic pairings $P_{i,f_{n-1}}$ on each of the $X_{i,f_{n-1}}$ above, and take
\begin{equation}\label{apareamientoXI}
P_{i}=\coprod_{f_{n-1}}P_{i,f_{n-1}}.
\end{equation}

Let $f_{n-1}\in L_i$, $i\in\Sigma_n$. 
\begin{itemize}\item If $i$ is the maximal label, set $P_i=P_{i,f_{n-1}}=\varnothing$, which is the only possible pairing on $X_{i,f_{n-1}}$ (and on $X_i$), in view of Corollary~\ref{corollary1}.
\item If $i$ is the minimal label, so that $V_{f_{n-1}}$ is defined, set
\begin{equation}\label{simini}
P_{i,f_{n-1}} =
\left\lbrace (f_{k}\setminus V_{f_{n-1}},f_{k}) \;\big\lvert\;  
f_{k}\in X_{i,f_{n-1}}\mbox{ with }V_{f_{n-1}}\in f_k\rule{.3mm}{0mm}\mbox{ and }\rule{.3mm}{0mm}k>1\right\rbrace.
\end{equation}
\item If $i$ is not the maximal or the minimal label, so that $V_{f_{n-1}}$ is defined, set 
\begin{equation}\label{nomini}
P_{i,f_{n-1}} =
\left\lbrace (f_{k}\setminus V_{f_{n-1}},f_{k}) \;\big\lvert\;  
f_{k}\in X_{i,f_{n-1}}\mbox{ with }V_{f_{n-1}}\in f_k\right\rbrace.
\end{equation}
\end{itemize}

\begin{proposition}\label{acyclicpairing}
For $f_{n-1}\in L_i$, $P_{i,f_{n-1}}$ is an acyclic pairing on $X_{i, f_{n-1}}$. Consequently~(\ref{apareamientoXI}) gives an acyclic pairing on $X_i$.
\end{proposition}
\begin{proof}
We only need to consider the case where $i$ is not the maximal label. Consider the coordinate $f_k\setminus V_{f_{n-1}}$ of any pair in~(\ref{simini}) or~(\ref{nomini}). First of all, $f_k\setminus V_{f_{n-1}}$ is non-empty as $V_{f_{n-1}}\in f_k\neq V_{f_{n-1}}$, where the latter inequality follows directly  from the requirement $k>1$ in the case of~(\ref{simini}), and from Proposition~\ref{vertex} in the case of~(\ref{nomini}). Next we argue that
\begin{equation}\label{paraargumentar}
\mbox{$f_k\setminus V_{f_{n-1}} \in X_{i,f_{n-1}}$:}
\end{equation}
By Theorem~\ref{minimalidad}, $f_{n-1}$ is the maximal flag determined by the minimal matrix representation of $f_{k}$. Say 
$$
f_{n-1}^i=V_1\subsetneq V_{2}\subsetneq\cdots\subsetneq V_{n-1} \text{ \ \ and \ \ }
f_{k}=V_{d_1}\subsetneq V_{d_2}\subsetneq\cdots\subsetneq V_{d_k},
$$
with $V_{f_{n-1}}=V_{d_\ell}$ for some $\ell=1,\ldots,k$ (i.e.~$d_\ell=j_i$, in the notation following Lemma~\ref{subcx}). Since $i_\ell<i_{\ell+1}$, by definition of $V_{f_{n-1}}$, Proposition~\ref{removing} gives~(\ref{paraargumentar}). 

So far we have made sure that $P_{i,f_{n-1}}$ is a subset of $X_{i,f_{n-1}}\times X_{i,f_{n-1}}$; the rest is easy (and standard). $P_{i,f_{n-1}}$ is a honest pairing since the first (second) coordinate in a pair in~(\ref{simini}) or~(\ref{nomini}) determines the second (first) coordinate, while a such first coordinate cannot appear also as a second coordinate (unlike the former ones, the latter ones use $V_{f_{n-1}}$ as a vertex). Lastly, $P_{i,f_{n-1}}$ is acyclic. Indeed, if $(f_{k-1},f_{k})\in P_{i,f_{n-1}}$ and $g_{k-1}\in X_{i,f_{n-1}}$ is a face of $f_{k}$ with $g_{k-1}\neq f_{k-1}$ then, by construction, there is no $g_{k}\in X_{i,f_{n-1}}$ with $(g_{k-1},g_{k})\in P_{i,f_{n-1}}$. In particular $P_{i,f_{n-1}}$ is acyclic.
\end{proof}

\begin{corollary}\label{flaghomotopytype}
$F(\mathbb{F}_q^n)$ has the homotopy type of a wedge of $(n-2)$-spheres.
\end{corollary}
\begin{proof}
By Lemma~\ref{lemmapairing}, $\bigcup_{i\in \Sigma_n} P_i$ is an acyclic pairing on $X$, and by Theorem~\ref{homotopytype}, $F(\mathbb{F}_q^n)$ is homotopy equivalent to a cell complex containing as many cells as the number of critical simplices.

By Corollary~\ref{corollary1}, if $i \in \Sigma_n$ is the maximal label, then $X_i$ consists {entirely} of critical faces, all of dimension $n-2$. On the other hand, if $i\in \Sigma_n$ is not the maximal element, each $f_k \in {X_{i,f_{n-1}}}$ is paired with either $f_k+ V_{f_{n-1}}$ or $f_k\setminus V_{f_{n-1}}$. The only possible exception is when $f_k= V_{f_{n-1}}$ which, as discussed in the first half of the proof of Proposition~\ref{acyclicpairing}, is an actual exception only when $i$ is the minimal label. Therefore, there is only one more critical simplex of dimension 0 which comes from $X_{1 2 \cdots n}$. (Note that there is a single maximal flag having label $1 2\cdots n$.)

Finally, a cell complex obtained from a point by attaching cells of a fixed dimension is a wedge of spheres of that dimension, in this case $n-2$.
\end{proof}

\begin{example} Coming back to Example~\ref{exampleF23} {for} $F(\mathbb{F}_2^3)$, we indicate the labels (with colors) in the Heawood graph (Figure~\ref{Heawood2}) as well as the pairings $P_i$ {(with arrows)} constructed above. Observe that the black vertex is the 0 critical simplex and the dark blue lines are the critical simplexes of dimension $n-2$ (here $n=3$) coming from the maximal label. Therefore $\displaystyle F(\mathbb{F}_2^3) \simeq \bigvee_{8} S^1$. 
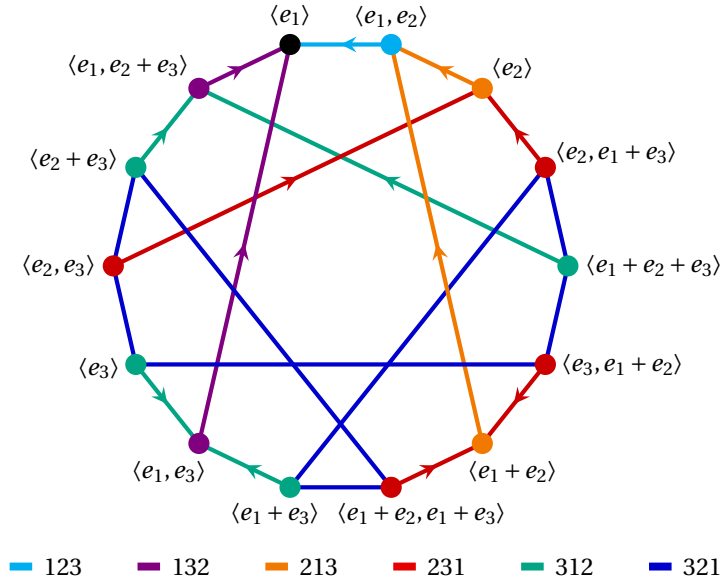
\begin{figure}[htb]
\begin{center}\footnotesize
\definecolor{mygreen}{HTML}{01a583}
\begin{tikzpicture}[very thick,decoration={
    markings,
    mark=at position 0.5 with {\arrow{stealth}}}
    ] 
\node[minimum size=6cm,regular polygon,regular polygon sides=14] (a) {};
\draw[ultra thick, mygreen, postaction={decorate}] (a.corner 12)node[right, xshift=0.1cm, black]{$\langle e_1+e_2+e_3\rangle$} -- (a.corner 3)node[left, yshift=0.3cm, black]{$\langle e_1, e_2+e_3\rangle$};
\draw[ultra thick, blue!80!black] (a.corner 13)node[right, yshift=0.2cm, black] {$\langle e_2, e_1+e_3\rangle$}--(a.corner 8)node[below,xshift=-0.2cm, yshift=-0.08cm, black]{$\langle e_1+e_3\rangle$};
\draw[ultra thick, red!80!black, postaction={decorate}] (a.corner 5)node[left, xshift=-0.1cm, black]{$\langle
 e_2, e_3 \rangle$}--(a.corner 14)node[above, xshift=0.4cm,black]{$\langle e_2 \rangle$};
\draw[ultra thick, orange!95!black, postaction={decorate}] (a.corner 10)node[below,xshift=0.4cm, yshift=-0.1cm,black]{$\langle
 e_1+e_2\rangle$}--(a.corner 1)node[above, yshift=0.1cm, black]{$\langle e_1,e_2 \rangle$};
\draw[ultra thick, violet, postaction={decorate}] (a.corner 7)node[below, xshift=-0.4cm, yshift=-0.1cm, black]{$\langle
 e_1,e_3 \rangle$}--(a.corner 2)node[above, yshift=0.12cm, black]{$\langle e_1 \rangle$};
\draw[ultra thick, blue!80!black] (a.corner 4)node[left, xshift=-0.1cm, yshift=0.1cm, black]{$\langle e_2+e_3\rangle$}--(a.corner 9)node[below, xshift=0.4cm, yshift=-0.08cm, black]{$\langle e_1+e_2, e_1+e_3\rangle$};
\draw[ultra thick, blue!80!black] (a.corner 6)node[left, xshift=-0.08cm, yshift=-0.05cm, black]{$\langle e_3 \rangle$}--(a.corner 11)node[right, xshift=0.08cm, black]{$\langle e_3, e_1+e_2 \rangle$};
\draw[ultra thick, blue!80!black] (a.corner 12)--(a.corner 13);
\draw[ultra thick, blue!80!black] (a.corner 11)--(a.corner 12);
\draw[ultra thick, blue!80!black] (a.corner 4)--(a.corner 5);
\draw[ultra thick, blue!80!black] (a.corner 5)--(a.corner 6);
\draw[ultra thick, blue!80!black] (a.corner 8)--(a.corner 9);
\draw[ultra thick, mygreen, postaction={decorate}] (a.corner 4)--(a.corner 3);
\draw[ultra thick, mygreen, postaction={decorate}] (a.corner 6)--(a.corner 7);
\draw[ultra thick, mygreen, postaction={decorate}] (a.corner 8)--(a.corner 7);
\draw[ultra thick, red!80!black, postaction={decorate}] (a.corner 13)--(a.corner 14);
\draw[ultra thick, red!80!black, postaction={decorate}] (a.corner 9)--(a.corner 10);
\draw[ultra thick, red!80!black, postaction={decorate}] (a.corner 11)--(a.corner 10);
\draw[ultra thick, orange!95!black, postaction={decorate}] (a.corner 14)--(a.corner 1);
\draw[ultra thick, violet, postaction={decorate}] (a.corner 3)--(a.corner 2);
\draw[ultra thick, cyan, postaction={decorate}] (a.corner 1)--(a.corner 2);
\fill[cyan] (a.corner 1) circle[radius=4pt];
\fill (a.corner 2) circle[radius=4pt];
\fill[violet] (a.corner 3) circle[radius=4pt];
\fill[mygreen] (a.corner 4) circle[radius=4pt];
\fill[red!80!black] (a.corner 5) circle[radius=4pt];
\fill[mygreen] (a.corner 6) circle[radius=4pt];
\fill[violet] (a.corner 7) circle[radius=4pt];
\fill[mygreen] (a.corner 8) circle[radius=4pt];
\fill[red!80!black] (a.corner 9) circle[radius=4pt];
\fill[orange!95!black] (a.corner 10) circle[radius=4pt];
\fill[red!80!black] (a.corner 11) circle[radius=4pt];
\fill[mygreen] (a.corner 12) circle[radius=4pt];
\fill[red!80!black] (a.corner 13) circle[radius=4pt];
\fill[orange!95!black] (a.corner 14) circle[radius=4pt];
\draw[cyan, line width=3pt] (-4.4,-4)--(-4.1,-4)node[black, right]{123};
\draw[violet, line width=3pt] (-2.7,-4)--(-2.4,-4)node[black, right]{132};
\draw[orange!95!black, line width=3pt] (-1,-4)--(-0.7,-4)node[black, right]{213};
\draw[red!90!black, line width=3pt] (0.7,-4)--(1,-4)node[black, right]{231};
\draw[mygreen, line width=3pt] (2.4,-4)--(2.7,-4)node[black, right]{312};
\draw[blue!80!black, line width=3pt] (4.1,-4)--(4.4,-4)node[black, right]{321};
\end{tikzpicture}
\end{center}
\caption{The labeling for $F(\mathbb{F}_2^3)$ and pairings indicated by arrows.}\label{Heawood2}
\end{figure}
\end{example}

\section{Counting Critical Cells}

Maximal flags in $F(\mathbb{F}^n_q)$ are in one-to-one correspondence with $n\times n$ matrices with entries in $ \mathbb{F}^n_q$ satisfying conditions \emph{(ii)} and \emph{(iii)} in Section~\ref{secpar}. Thus, the counting principle easily gives that the number $f^i$ of maximal flags having a given label $i\in\Sigma_n$ is $f^i=q^{\sum_{j=1}^n m_j}$,
where
$$
m_j = i_j-1-\left\lvert\{i_k \,|\, i_k < i_j \text{ and } k<j \}\right\rvert
$$
and $i=i_1 i_2 \cdots i_n$. Indeed, the summand ``${}-1$'' comes from condition~\emph{(ii)}, and the summand ``${}-\left\lvert\{i_k \,|\, i_k < i_j \text{ and } k<j \}\right\rvert$'' comes from condition~\emph{(iii)}. In particular, there exists $f^{n(n-1) \cdots 1}= q^{\sum_{k=1}^{n-1} k} = q^{\binom{n}{2}}$ maximal flags with maximal label.

\begin{corollary}\label{betticompletos}
For $n\geq 2$, the flag complex $F(\mathbb{F}_q^n)$ has the homotopy type of the wedge of $q^{\binom{n}{2}}$ spheres of dimension $n-2$.
\end{corollary}

\begin{example}
For $n=2$, the homotopy equivalence in Corollary~\ref{betticompletos} is in fact a homeomorphism: $\mathbb{F}_q^2$ has precisely $q+1$ lines.
\end{example}

\begin{remark}\label{alesqueleto}
Deleting all pairs in the pairing $P = \bigcup_{i \in \Sigma_n} P_i$ having simplices of dimension greater than $k$, yields a pairing for the $k$-skeleton of $F(\mathbb{F}_q^n)$, which remains acyclic because $P$ is acyclic in view of Lemma~\ref{lemmapairing} and Proposition~\ref{acyclicpairing}. 
\end{remark}

\begin{proposition}\label{kskeleton}
The $k$-skeleton of $F(\mathbb{F}_q^n)$ has the homotopy type of a wedge of $k$-spheres.
\end{proposition}
\begin{proof}
After deleting from $P$ those pairs having simplices of dimension greater than $k$, we get an acyclic paring (Remark~\ref{alesqueleto}) for the $k$-skeleton of $F(\mathbb{F}^n_q)$ having some critical simplices $f_{k+1}$ of dimension $k$ ({namely, those $f_{k+1}$  with $(f_{k+1}, f_{k+1}+V_{f_{n-1}})\in P_{i,f_{n-1}}$ for some maximal flag $f_{n-1}$ with label $i$}) and only one critical $0$-simplex. Theorem~\ref{homotopytype} implies that the $k$-skeleton has the homotopy type of a wedge of $k$-spheres.
\end{proof}

An ascending pair of a label $i=i_1i_2\cdots i_n\in\Sigma_n$ is a pair $(i_t,i_{t+1})$ of consecutive indices with $i_t<i_{t+1}$. For instance, any non-maximal label $i$ has at least one ascending pair (the one with $t=j_i$, in the notation following Lemma~\ref{subcx}). Let $p_i$ denote the number of ascending pairs in the label~$i$. 

Note that the critical $k$-simplices in the previous proof occur when we can remove $n-2-(k+1)$ vertices from the set of ascending pairs of $i$ neither of which is the initial ascending pair $(i_{j_i},i_{j_i+1})$ (by Proposition~\ref{removing}, the face obtained in this fashion is also in $X_i$). This allows us to count the number of $k$-spheres in the {wedge sum} of Proposition~\ref{kskeleton}. Recall the cardinality of $L_i$, $f^i=\lvert L_i \rvert$, is determined at the beginning of this section.

\begin{corollary}{For $k\in \{0,...,n-2\}$, the $k$-th skeleton of $F(\mathbb{F}_q^n)$ has the homotopy type of a wedge of $\displaystyle\sum_{i\in \Sigma_{n}}\Big(\begin{smallmatrix} p_i-1 \\[2pt] n-k-3 \end{smallmatrix}\Big)\cdot f^i$ spheres of dimension $k$.}
\end{corollary}

\bigskip\smallskip
{\small \sc Departamento de Matem\'aticas \\
Centro de Investigaci\'on y de Estudios Avanzados del I.P.N. \\
Av.~Instituto Polit\'ecnico Nacional n\'umero 2508 \\
San Pedro Zacatenco, M\'exico City 07000, M\'exico \\
{\tt jaguzman@math.cinvestav.mx} \\
{\tt jesus@math.cinvestav.mx} \\
{\tt jlleon@math.cinvestav.mx}
}

\begin{thebibliography}{1}

\bibitem{Forman}
Robin Forman.
\newblock A user's guide to discrete {M}orse theory.
\newblock {\em S\'em. Lothar. Combin.}, 48:Art. B48c, 35, 2002.

\bibitem{hersh}
Patricia Hersh.
\newblock On optimizing discrete {M}orse functions.
\newblock {\em Advances in Applied Mathematics}, 35(3):294--322, 2005.

\bibitem{MR2724673}
J\"{u}rgen Herzog and Takayuki Hibi.
\newblock {\em Monomial ideals}, volume 260 of {\em Graduate Texts in
  Mathematics}.
\newblock Springer-Verlag London, Ltd., London, 2011.

\bibitem{jonsson}
Jakob Jonsson.
\newblock {\em Simplicial complexes of graphs}, volume 1928 of {\em Lecture
  Notes in Mathematics}.
\newblock Springer-Verlag, Berlin, 2008.

\bibitem{MR2361455}
Dmitry Kozlov.
\newblock {\em Combinatorial algebraic topology}, volume~21 of {\em Algorithms
  and Computation in Mathematics}.
\newblock Springer, Berlin, 2008.

\bibitem{Kratzer-Thevenaz}
Charles Kratzer and Jacques Th\'evenaz.
\newblock Type d'homotopie des treillis et treillis des sous-groupes d'un
  groupe fini.
\newblock {\em Comment. Math. Helv.}, 60(1):85--106, 1985.

\bibitem{Zax}
Rachel~Elana Zax.
\newblock Simplifying complicated simplicial complexes: Discrete {M}orse theory
  and its applications, {A}.{B}. {T}hesis, Harvard University, 2012.
\newblock Available from:
  \url{http://www.math.harvard.edu/theses/senior/zax/zax.pdf}.

\end{thebibliography}
\end{document}